\documentclass{amsart}
\usepackage{amssymb}

\usepackage{mathtools}

\mathtoolsset{showonlyrefs} %	enumera solo las ecuaciones que se citan

\usepackage{amsfonts}
\usepackage{color}
\usepackage{graphicx,tikz}
\usepackage[bookmarksnumbered,colorlinks, linkcolor=blue, citecolor=red, pagebackref, bookmarks, breaklinks]{hyperref}

\newtheorem{thm}{Theorem}[section]
\newtheorem{cor}[thm]{Corollary}
\newtheorem{conj}[thm]{Conjecture}

\newtheorem{prop}[thm]{Proposition}
\theoremstyle{definition}

\theoremstyle{remark}

\newtheorem{rem}[thm]{Remark}
 
\newcommand{\R}{\mathbb R}
\newcommand{\N}{\mathbb N}

\newcommand{\lam}{\lambda}

\author[E. Parini]{Enea Parini}
\author[J. D. Rossi]{Julio D. Rossi}
\author[A. Salort]{Ariel Salort}

\address[E. Parini]{Aix Marseille Univ, CNRS, Centrale Marseille, I2M, 39 Rue Fr\'{e}d\'{e}ric Joliot Curie, 13453 Marseille CEDEX 13, France}
\email{enea.parini@univ-amu.fr}

\address[J. D. Rossi and A. Salort]{IMAS-Conicet and 
Departamento  de Matem{\'a}tica, FCEyN, Universidad de Buenos Aires,
Ciudad Universitaria. Pab 1, (1428)
Buenos Aires,
Argentina.} \email{jrossi@dm.uba.ar, asalort@dm.uba.ar}

\title[Reverse Faber-Krahn inequality]{Reverse Faber-Krahn inequality for a truncated laplacian operator}


\begin{document}

\begin{abstract} In this paper we prove a reverse Faber-Krahn inequality for the principal eigenvalue $\mu_1(\Omega)$ of the fully nonlinear eigenvalue problem \[ \label{eq}
\left\{\begin{array}{r c l l} 
-\lambda_N(D^2 u) & = & \mu u & \text{in }\Omega, \\ u & = & 0 & \text{on }\partial \Omega. 
\end{array}\right.
\]
Here $ \lambda_N(D^2 u)$ stands for the largest eigenvalue of the Hessian matrix of $u$. More precisely, we prove that, for an open, bounded, convex domain $\Omega \subset \R^N$, the inequality
\[ \mu_1(\Omega) \leq \frac{\pi^2}{[\text{diam}(\Omega)]^2} = \mu_1(B_{\text{diam}(\Omega)/2}),\]
where $\text{diam}(\Omega)$ is the diameter of $\Omega$, holds true. 
The inequality actually implies a stronger result, namely, the maximality of the ball under a diameter constraint. 

Furthermore, we discuss the minimization of $\mu_1(\Omega)$ under different kinds of constraints.
\end{abstract}

\maketitle

\section{Introduction}
Given a bounded domain $\Omega\subset \R^N$, we consider the fully nonlinear eigenvalue problem 
\begin{equation} \label{eq.intro}
\left\{\begin{array}{ l l} 
-\lambda_N(D^2 u)  =  \mu u & \text{in }\Omega, \\ u  =  0 & \text{on }\partial \Omega, 
\end{array}\right.
\end{equation}
where $\mu \in \R$ is the eigenvalue, $D^2 u$ is the Hessian matrix of a function $u: \Omega \to \R$, and $\lambda_1(D^2 u),\dots,\lambda_N(D^2 u)$ are the ordered eigenvalues of $D^2 u$, so that
\[ 
\lambda_1(D^2 u) \leq \dots \leq \lambda_N(D^2 u). 
\]
 
Following the celebrated approach by Berestycki, Varadhan, and Nirenberg \cite{BNV}, it is possible to define the first eigenvalue $\mu_1(\Omega)$ as
\begin{equation} \label{definitioneigenvalue}
\mu_1(\Omega) := \sup \{ \mu \in \R \colon \exists \varphi \in LSC(\Omega),\, \varphi > 0 \text{ in }\Omega \text{ s.t. } -\lambda_N(D^2 \varphi) \geq \mu \varphi\},
\end{equation}
where $LSC(\Omega)$ is the set of lower semicontinuous functions in $\Omega$, and the inequality $ -\lambda_N(D^2 \varphi) \geq \mu \varphi$ is understood in the viscosity sense.
When the domain $\Omega$ is assumed to be strictly convex, the existence of a strictly positive eigenfunction associated to $\mu_1(\Omega)$ 
has been proven in \cite{BGI}. It is still unknown whether the assumptions on the domain can be weakened to (not strict) convexity; however, it should be observed that, if $u$ is a positive eigenfunction, then it holds, in the viscosity sense,
 \[ 
\lambda_1(D^2u) \leq \dots \leq \lambda_N(D^2 u) \leq 0,
\]
which implies the negative semidefiniteness of $D^2 u$, and hence the concavity of $u$. The fact that $\Omega = \{u > 0\}$ implies the convexity of $\Omega$. For this reason, unless otherwise stated, we will always suppose that $\Omega$ is a convex domain.

The operator $\mathcal{P}_N^+ (D^2 u):= \lam_N(D^2 u)$ belongs to the class of so-called \emph{truncated Laplacian} operators, 
which are defined as the sum of a finite number of consecutive eigenvalues of $D^2 u$. More precisely, this class of operators is defined, 
for $k=1,\dots,N$, as
\[ \mathcal{P}_k^+ (D^2 u) := \sum_{i=N-k+1}^{N} \lambda_i(D^2 u),\qquad\mathcal{P}_k^- (D^2u) := \sum_{i=1}^{k} \lambda_i(D^2 u)\]
and received some attention recently, see \cite{BGI,BGI2,BR,BR2,BER,Caffa,HL,HL2}.
These highly degenerate elliptic operators appeared first in the context of differential geometry, see \cite{AS,Sha,Wu}

A natural question arising in spectral theory is the following: which domain minimizes or maximizes the first eigenvalue of a differential operator, under some geometrical constraint on the domain? A classical example is given by the Laplacian operator
\[
\Delta u = \text{tr}\, (D^2 u)= \sum_{i=1}^N \lam_i(D^2 u).
\]
In this case, the well-known \emph{Rayleigh-Faber-Krahn inequality} states that the first eigenvalue $\mu_1^\Delta$ of $-\Delta$ under Dirichlet boundary conditions is uniquely minimized, among domains with fixed volume, by the ball. More precisely, if $\Omega \subset \R^N$, and $B_r$ is a ball of radius $r$ such that $|\Omega|=|B_r|$, then
$$
\mu_1^\Delta(B_r) \leq \mu_1^\Delta(\Omega),
$$
and equality holds if and only if $\Omega=B_r$.

However, if we consider the \emph{Monge-Amp\`{e}re operator}, which is the fully nonlinear operator defined as
\[\text{det}(D^2 u) = \prod_{i=1}^N \lambda_i(D^2 u),\]
the situation is completely different. Indeed, among all bounded, convex domains with fixed volume, the (unique) eigenvalue is \emph{maximized} by the ball \cite[Theorem 1.4]{le}, while it is conjectured that the $N$-dimensional regular simplex is a minimizer. 

Concerning the truncated Laplacian operator $\mathcal{P}_N^+$, Birindelli, Galise and Ishii conjectured in \cite{BGI2} the validity of a reverse Faber-Krahn inequality, namely, that the ball maximizes $\mu_1(\Omega)$ among convex domains with fixed volume. The conjecture was supported by some partial results proven in \cite{BGI2}: 
\begin{itemize}
\item The hypercube has the largest first eigenvalue among hyperrectangles of given measure; 
\item The ball has a larger first eigenvalue than the hypercube with the same volume.
\end{itemize}

In this paper we prove the validity of the conjectured reverse Faber-Krahn inequality. 
More precisely, we first prove 
the following theorem. 

\begin{thm} \label{teo.intro.max}
For an open, bounded, convex domain $\Omega \subset \R^N$, the inequality
\[ \mu_1(\Omega) \leq \frac{\pi^2}{[\text{diam}(\Omega)]^2},\]
where $\text{diam}(\Omega)$ is the diameter of $\Omega$, holds true.
\end{thm}

Theorem \ref{teo.intro.max} readily implies that the ball maximizes $\mu_1(\Omega)$ under a diameter constraint, and, as a consequence, we obtain the reverse Faber-Krahn inequality.

\begin{thm}[\bf{Reverse Faber-Krahn inequality}] \label{teo.intro.faber}
Let $\Omega \subset \R^N$ be an open, bounded, convex set. Let $B_r$ be a ball such that $|\Omega|=|B_r|$. Then,
\[ \mu_1(\Omega) \leq \mu_1(B_r),\]
and equality holds if and only if $\Omega = B_r$.
\end{thm}

Although the techniques we employed to prove these theorems are of rather elementary nature, we believe that the results are among the first examples of isoperimetric inequalities for fully nonlinear operators, apart from the case of the well-known Monge-Amp\`{e}re operator.

Finally, let us briefly discuss the \emph{minimization} problem for $\mu_1(\Omega)$. We will see in Section \ref{preliminary} that the minimization problem for $\mu_1(\Omega)$ under a volume constraint does not admit a solution, since there exists a sequence of hyperrectangles $\{R_n\}_{n \in \N}$ 
that degenerate to a line keeping fixed volume such that $\mu_1(R_n) \to 0$ as $n \to +\infty$. Nevertheless, it is possible to consider the same minimization problem under a perimeter or a diameter constraint. In Section \ref{minimizationsection} we provide some preliminary results and we state some conjectures.

\subsection*{Acknowledgments} 
This article was started during a visit of the third author to Aix Marseille University in the context of the MSCA Project GHAIA (ref:777822).

\section{Preliminary results} \label{preliminary}

\subsection*{Notations}

For a measurable set $\Omega \subset \R^N$, we will denote by $|\Omega|$ its volume, which corresponds to its $N$-dimensional Lebesgue measure. $P(\Omega)$ will stand for the perimeter of $\Omega$; since we will be mainly dealing with convex sets, we can define $P(\Omega)$ as
\[ P(\Omega) := \mathcal{H}^{N-1}(\partial \Omega),\]
where $\mathcal{H}^{N-1}$ is the $(N-1)$-dimensional Hausdorff measure. The diameter of $\Omega$ will be denoted by $\text{diam}(\Omega)$,
\[ \text{diam}(\Omega):= \sup \{ |x-y|\,|\,x,y \in \Omega\}.\]
For $r>0$, the symbol $B_r$ will stand for an open ball of radius $r$.

\subsection*{The eigenvalue problem}

Let $\Omega \subset \R^N$ be an open, bounded domain. It is clear from the definition of $\mu_1(\Omega)$ in \eqref{definitioneigenvalue} that the function
\[ \Omega \mapsto \mu_1(\Omega) \]
is monotone decreasing with respect to set inclusion, that is,
\[ \Omega_1 \subset \Omega_2 \Rightarrow \mu_1(\Omega_1) \geq \mu_1(\Omega_2).\]
Moreover, $\mu_1(\Omega)$ enjoys the following scaling property: for a given open, bounded set $\Omega \subset \R^N$, and $t>0$, define
\[ t\Omega := \{ x \in \R^N \,|\,x/t \in \Omega \}.\]
Then,
\[ \mu_1(t\Omega) = \frac{1}{t^2}\mu_1(\Omega).\]
When $\Omega \subset \R^N$ is a strictly convex domain, it is possible to show the existence of a strictly positive viscosity solution $u \in C(\overline{\Omega})$ to the eigenvalue problem
\begin{equation} \label{eq2}
\left\{\begin{array}{ l l} 
-\lambda_N(D^2 u)  =  \mu u & \text{in }\Omega, \\ u  =  0 & \text{on }\partial \Omega, 
\end{array}\right.
\end{equation}
for $\mu = \mu_1(\Omega)$ (see \cite{BGI2}). 
It is still an open problem to determine whether also (not strictly) convex domains admit a positive first eigenfunction. 
However, convexity turns out to be a necessary condition. If $u \in C(\overline{\Omega})$ is a positive eigenfunction, then it satisfies, in viscosity sense,
$ \lambda_1(D^2u) \leq \dots \leq \lambda_N(D^2 u) \leq 0$ and hence $u$ is concave. 
Hence, the fact that $\Omega = \{u > 0\}$ implies the convexity of $\Omega$. 

For some particular domains, it is possible to compute the eigenvalue $\mu_1(\Omega)$ and its associated eigenfunction (see \cite{BGI2}):
\begin{itemize}
\item 
When $\Omega=B_r\subset \R^N$ is the ball of radius $r$, it holds
\begin{equation} \label{eig.ball}
\mu_1(B_r) = \frac{\pi^2}{4r^2}
\end{equation}
and the associated eigenfunction is given by
\begin{equation}
u(x)=\cos\left(\frac{\pi}{2r}|x|\right).
\end{equation}
It is worth noticing that the eigenvalue and the eigenfunction do not depend on the space dimension $N$.
\item In the hyperrectangle $R = \prod_{i=1}^N \left(-\alpha_i,\alpha_i\right)\subset \R^N$ the first eigenvalue is given by
\begin{equation} \label{eig.cube}
\mu_1(R) = \frac{\pi^2}{4(\alpha_1^2 + \dots + \alpha_N^2)}
\end{equation}
and the associated eigenfunction has the form
\[ u(x) = \prod_{i=1}^N \left[\cos{\left( \frac{\pi}{2\alpha_i}x_i\right)}\right]^{\frac{1}{p_i + 1}}\]
for some adequate values of $p_i > -1$.
\end{itemize}

These examples have some significant consequences. First of all, it follows by monotonicity that the quantity $\mu_1(\Omega)$ is well defined and finite for every open, bounded set, since there exists a sufficiently small ball $B_r \subset \Omega$, so that
\[ \mu_1(\Omega) \leq \mu_1(B_r) < + \infty.\]

Moreover, the case of the hyperrectangle shows that the problem of minimizing $\mu_1(\Omega)$ under a volume constraint 
does not have a solution: Fix $c>0$, and take $\{R_n\}_{n \in \N}$ is a sequence of hyperrectangles defined as
\[ R_n = (0,n) \times \underbrace{ (0,1) \times \dots \times (0,1) }_{(N-2) \text{ times} }\times  \left(0,\frac{c}{n}\right).\]
Then, it is straightforward to verify that $|R_n|=c$ for every $n$ and 
\[ \lim_{n \to +\infty} \mu_1(R_n) = 0.\]

\section{Maximization of the first eigenvalue under constraint}

In this section we will deal with the problem of maximizing the first eigenvalue $\mu_1(\Omega)$ among convex sets, under various 
different constraints.

The key result is the following estimate for $\mu_1(\Omega)$, which is reminiscent of a similar result (with reversed inequality) obtained by Payne and Weinberger for the first nontrivial eigenvalue of the Laplacian under Neumann boundary conditions, see \cite{payneweinberger}.

\begin{prop} \label{isodiametric}
Let $\Omega \subset \R^N$ be an open, bounded, convex domain. Then,
\begin{equation} \label{eq:isodiametric} \mu_1(\Omega) \leq \frac{\pi^2}{[\text{diam}(\Omega)]^2},\end{equation}
where $\text{diam} (\Omega)$ is the diameter of $\Omega$.
\end{prop}
\begin{proof}
Let $\{x_n\}_{n\in\mathbb{N}}$ and $\{y_n\}_{n\in\mathbb{N}}$ be two sequences of points in $\Omega$ such that
\[ d_n:= \text{dist}(x_n,y_n) \to \text{diam} (\Omega) \qquad \text{as }n \to +\infty.\]
By convexity, $\Omega$ contains, for $\varepsilon_n \in \left(0, \frac{1}{n}\right)$ sufficiently small, a hyperrectangle $R$ of sides 
of length $\varepsilon_n$ ($N-1$ times) and $d_n$ (one time). By monotonicity of $\mu_1$, we obtain
 \[ \mu_1(\Omega) \leq \mu_1(R) = \frac{\pi^2}{(N-1)\varepsilon_n^2 + d_n^2.}\]
Passing to the limit $n \to +\infty$, we obtain
\[ \mu_1(\Omega) \leq \frac{\pi^2}{[\text{diam} (\Omega)]^2}\]
as we wanted to show.
\end{proof}

As an immediate corollary we have the following.

\begin{cor} \label{corol}
Let $\Omega \subset \R^N$ be a convex set which is contained in a ball of radius $\frac{\text{diam} (\Omega)}{2}$. Then,
\begin{equation} \label{alwaysequal} \mu_1(\Omega) = \frac{\pi^2}{[\text{diam} (\Omega)]^2}.\end{equation}
\end{cor}
\begin{proof}
The result follows by the monotonicity of $\mu_1(\Omega)$ with respect to set inclusion, Proposition \ref{isodiametric} and the explicit expression 
of $\mu_1(\Omega)$ for the ball.
\end{proof}

Examples of convex sets satisfying the conditions of Corollary \ref{corol} include regular polygons in $\R^2$ with an even number of sides, or hypercubes in $\R^N$. 
However, not every convex set enjoys this property: for instance, an equilateral triangle can not be contained in a disk of radius smaller than $\frac{\text{diam} (\Omega)}{\sqrt{3}}$. Nevertheless, by Jung's Theorem \cite{jung} we obtain the following lower bound:
\begin{equation} \label{lowerboundjung} \mu_1(\Omega) \geq \frac{N+1}{2N} \cdot \frac{\pi^2}{[\text{diam} (\Omega)]^2},
\end{equation}
for any open, bounded, convex set $\Omega \subset \R^N$.

In view of these considerations, one might wonder whether there actually exist domains such that   inequality \eqref{eq:isodiametric} is strict. In the next proposition we show, under a smoothness assumption on the eigenfunction, that this is indeed the case for the \emph{Reuleaux triangle}. We observe that, if $\mathcal{R} \subset \R^2$ is the Reuleaux triangle generated by the open equilateral triangle $T \subset \mathcal{R}$ of unit side length, then it holds $\text{diam}(\mathcal{R})=1$.

\begin{figure}
\centering
\includegraphics[width=6cm]{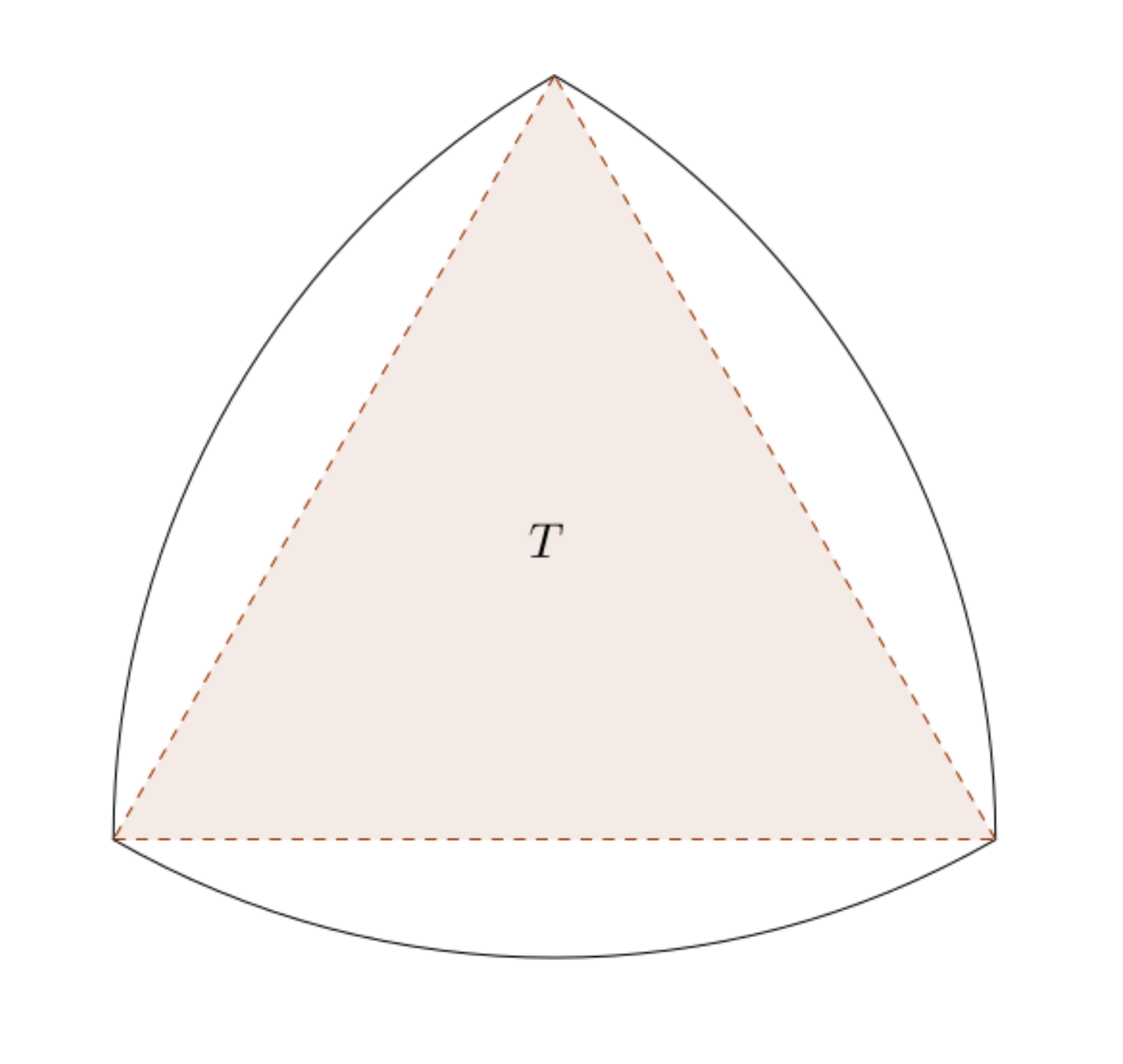} 
\caption{The Reuleaux triangle generated by an equilateral triangle}
\end{figure}

\begin{prop} \label{reuleaux}
 Let $\mathcal{R} \subset \R^2$ be the Reuleaux triangle generated by the open equilateral triangle $T \subset \mathcal{R}$ of unit side length. Suppose that the first eigenfunction is in $C^2(\mathcal{R})\cap C(\overline{\mathcal{R}})$. Then,
 \[ \mu_1(\mathcal{R}) < \pi^2.\]
\end{prop}
\begin{proof}
Let $u \in C^2(\mathcal{R})\cap C(\overline{\mathcal{R}})$ be a positive eigenfunction associated to the first eigenvalue $\mu_1(\mathcal{R})$. We know, by Proposition \ref{isodiametric}, that $$\mu_1(\mathcal{R})\leq \pi^2.$$ 
Suppose by contradiction that $\mu_1(\mathcal{R})=\pi^2$. Let $A$, $B$ and $C$ be the vertices of $T$, and let $P \in T$ be a point. 
Let $D \in \partial \mathcal{R}$ be the intersection of the line passing through $A$ and $P$ with $\partial \mathcal{R}$. Let $AD$ be the segment, of unit length, parametrized as
\[ AD= \{ x \in \overline{\mathcal{R}}\,|\,x= tD + (1-t)A,\,t \in [0,1]\}. \]
Set $e_1 := AD \in S^1$. Let $v$ be the restriction of $u$ to the segment $AD$, namely,
\[ v(t) := u(tD+(1-t)A).\]
The function $v$ satisfies $v(0)=v(1)=0$, and
\begin{align} \label{1Dsuper} 
\begin{split}
-v''(t) & = -\langle D^2 u(tD+(1-t)A) \cdot e_1, e_1 \rangle 
\\ & \geq -\lambda_2(D^2 u)(tD+(1-t)A) \\ & = \pi^2 u(tD+(1-t)A) = \pi^2 v(t).
\end{split}
\end{align}
From \cite[Theorem 2.2]{BNV} (applied to the operator $F=\Delta + \lambda_1$), $v$ coincides with a multiple of the eigenfunction in $(0,1)$, namely $v(t)=c \sin{(\pi t)}$ for some $c>0$. Therefore, equality holds in \eqref{1Dsuper}, which implies, in particular,
\[ -\langle D^2 u(P) \cdot e_1, e_1 \rangle = \pi^2 u(P),\]
$e_1$ being an eigenvector associated to $\lambda_2(D^2 u(P))$. Repeating the same reasoning with the line passing through $B$ and $P$, we would obtain that, for a unit vector $e_2 \in S^1$ different from $e_1$, it holds
\[ -\langle D^2 u(P) \cdot e_2, e_2 \rangle = \pi^2 u(P). \]
Therefore,
\[ -D^2 u(P) = \pi^2 u(P) I\]
for every $P \in T$. The conditions $$-u_{xx}=-u_{yy}=\pi^2 u$$ imply that $u$ is of the form
\[ u(x,y) = (a_1 \sin{\pi x}+b_1\cos{\pi x})(a_2 \sin{\pi y}+b_2\cos{\pi y}).\]
On the other hand, the condition $$u_{xy}=0$$ implies that $$a_1=a_2=b_1=b_2=0,$$ a contradiction to the fact that $u>0$ in $T$.
\end{proof}

\begin{rem}
In the proof of Proposition \ref{reuleaux}, we made the assumption that the eigenfunction belongs to $C^2(\mathcal{R})\cap C(\overline{\mathcal{R}})$. While the few explicit examples currently known (balls and hyperrectangles) support the conjecture that eigenfunctions belong to $C^\infty(\Omega) \cap C(\overline{\Omega})$, the optimal regularity of eigenfunctions of the truncated Laplacian is far from being understood, due to the high degeneracy of the differential operator. It is worth mentioning that, in the case of the Monge-Amp\`{e}re operator, which is also fully nonlinear, but enjoys particular structural properties, eigenfunctions belong to $C^\infty(\Omega) \cap C^{0,\beta}(\overline{\Omega})$ for every $\beta \in (0,1)$ (see \cite[Theorem 1.1]{le}). 
\end{rem}

Proposition \ref{isodiametric} readily implies that the ball of radius $\frac{d}{2}$ maximizes $\mu_1(\Omega)$ among convex domains with fixed diameter $\text{diam}(\Omega) = d>0$. The maximizer is not unique, since every other convex set with diameter $d$, contained in $B_{\frac{d}{2}}$, has the same eigenvalue.

We will denote by $\mathcal{K}_N$ the set 
\[ \mathcal{K}_N:= \Big\{ \Omega \subset \R^N\,|\,\Omega \text{ open, bounded and convex}\Big\}.\]

\begin{prop} \label{maxdiameter} Fix $d>0$. Then, 
the ball is a solution of the maximization problem
\[ \sup \Big\{\mu_1(\Omega)\,|\,\Omega \in \mathcal{K}_N,\,\text{diam}(\Omega)=d \Big\}.\]
\end{prop}

\begin{proof}
By Proposition \ref{isodiametric} and \eqref{eig.ball} we have
\[ \mu(\Omega) \leq \frac{\pi^2}{d^2} = \mu(B_{\frac{d}{2}})\]
for every convex domain $\Omega$ with $\text{diam}(\Omega)=d$.
\end{proof}

Proposition \ref{maxdiameter} directly implies that the ball maximizes $\mu_1(\Omega)$ under a perimeter or a volume constraint.

\begin{prop} \label{maxperimeter}
The ball is a solution of the maximization problem
\[ \sup \Big\{\mu_1(\Omega)\,|\,\Omega \in \mathcal{K}_N,\,P(\Omega)=c\Big\},\]
where $c>0$ is fixed. Moreover, if $N \geq 3$, then the ball is the unique maximizer.
\end{prop}
\begin{proof}
The proof is a consequence of Proposition \ref{maxdiameter}, and the fact that the ball maximizes perimeter under a diameter constraint among convex sets, being the unique maximizer if $N \geq 3$ (see \cite[Theorem 5]{maggiponsiglionepratelli}).
\end{proof}

\begin{prop}[{\bf Reverse Faber-Krahn inequality}] \label{maxvolume}
For $c>0$
the ball is the unique maximizer of the maximization problem
\[ \sup \Big\{\mu_1(\Omega)\,|\,\Omega \in \mathcal{K}_N,\,|\Omega|=c \Big\}.\] 
\end{prop}
\begin{proof}
The proof follows from Proposition \ref{maxperimeter} and the well-known isoperimetric property of the ball (see, for instance, \cite[Theorem 14.1]{maggi}). 
\end{proof}

\section{Minimization of the first eigenvalue under constraints} \label{minimizationsection}

In Section \ref{preliminary} we have seen that the problem of minimizing $\mu_1(\Omega)$ 
among bounded, convex sets of fixed volume 
does not have a solution. This holds since a minimizing sequence is given by a sequence of hyperrectangles whose diameter tends to infinity. Then, it is natural to wonder whether it makes sense to consider the minimization problem under other kinds of constraints. In this section we will mainly restrict to planar, convex sets, $\Omega \subset \mathbb{R}^2$, and we will provide some results which support the following conjectures.

\begin{conj}
The minimization problem
\begin{equation} \label{minperimeter} \inf \Big\{\mu_1(\Omega)\,|\,\Omega \in \mathcal{K}_2,\,P(\Omega)=c\Big\} \qquad (c > 0)\end{equation}
does not admit a solution. A minimizing sequence is given by a sequence of rectangles of constant perimeter, with one side length tending to zero.
\end{conj}

\begin{conj}
The minimization problem
\begin{equation} \label{mindiameter} \inf \Big\{\mu_1(\Omega)\,|\,\Omega \in \mathcal{K}_2,\,\text{diam}(\Omega)=d \Big\} \qquad (d > 0)\end{equation}
admits a solution.
\end{conj}

\begin{conj}
The Reuleaux triangle is a minimizer for Problem \eqref{mindiameter} 
\end{conj}

We refer to \cite[Chapter 2.2]{henrotpierre} for the basic definitions and notions that we will use in this section. Our first result is the continuity of $\Omega \mapsto \mu_1 (\Omega)$ with respect to the Hausdorff convergence of open sets.

\begin{prop} \label{continuityhausdorff}
Let $\{\Omega_n\}$ be a sequence of nonempty, open convex sets which converge, with respect to the Hausdorff convergence of open sets, to the nonempty, open convex set $\Omega \in \mathcal{K}_N$. Then,
\[ \lim_{n \to +\infty} \mu_1(\Omega_n) = \mu_1(\Omega).\]
\end{prop}

\begin{proof}
Let $\{\Omega_n\}_{n\in\N}$ be a sequence in $\mathcal{K}_N$ which converges, with respect to the Hausdorff convergence of open sets, to $\Omega \in \mathcal{K}_N$. Then, there exists a sequence $\{\varepsilon_n\}$ with $\varepsilon_n \to 0$ such that
\[ (1-\varepsilon_n)\Omega \subset \Omega_n \subset (1+\varepsilon_n)\Omega.\]
(see, for instance, \cite[p. 359]{BBI}). Due to the monotonicity and the scaling properties of $\mu_1$, we have
\[ \frac{1}{(1+\varepsilon_n)^2} \mu_1(\Omega) \leq \mu_1(\Omega_n) \leq \frac{1}{(1-\varepsilon_n)^2} \mu_1(\Omega).\]
This implies
\[ \lim_{n \to +\infty} \mu_1(\Omega_n) = \mu_1(\Omega),\]
which is the claim.
\end{proof}

The next result concerns the asymptotic behaviour of a shrinking sequence of convex, planar sets.

\begin{prop} \label{shrinking}
Let $\{\Omega_n\}_{n \in \N}$ be a sequence of convex, open sets in $\R^2$ such that $\text{diam}(\Omega_n) \to d > 0$ and $|\Omega_n| \to 0$ as $n \to +\infty$. Then,
\[ \lim_{n \to +\infty} \mu_1(\Omega_n) = \frac{\pi^2}{d^2}.\]
\end{prop}
\begin{proof}
Set $d_n := \text{diam}(\Omega_n)$. Without loss of generality, we can translate and rotate the sets $\Omega_n$ in such a way that the points $(0,0)$ and $(d_n,0)$ are on $\partial \Omega_n$. Let
\[ \varepsilon_n^{(1)} := \max \{ y \geq 0\,|\,(x,y) \in \partial \Omega_n\},\]
\[ \varepsilon_n^{(2)} := \min \{ y \leq 0\,|\,(x,y) \in \partial \Omega_n\}.\]
Since the sets $\Omega_n$ are convex, and $|\Omega_n|\to 0$ as $n \to +\infty$, it must hold $\varepsilon_n^{(1)}$, $\varepsilon_n^{(2)} \to 0$ as $n \to +\infty$. Moreover, $\Omega_n$ is contained in a rectangle $R_n$ of sides $d_n$ and $\varepsilon_n^{(1)} +\varepsilon_n^{(2)}$. By monotonicity,
\[ \mu_1(\Omega_n) \geq \mu_1(R_n) = \frac{\pi^2}{d_n^2 + (\varepsilon_n^{(1)} +\varepsilon_n^{(2)})^2} \Rightarrow \liminf_{n \to +\infty} \mu_1(\Omega_n) \geq \frac{\pi^2}{d^2}.\]
On the other hand, by Proposition \ref{isodiametric},
\[ \mu_1(\Omega_n) \leq \frac{\pi^2}{d_n^2} \Rightarrow \limsup_{n \to +\infty} \mu_1(\Omega_n) \leq \frac{\pi^2}{d^2}.\]
Hence, we have obtained that
\[ \lim_{n \to +\infty} \mu_1(\Omega_n) = \frac{\pi^2}{d^2},\]
as we wanted to show.
\end{proof}

Let us now discuss the minimization problem in $\R^2$
\[ \inf \Big\{\mu_1(\Omega)\,|\,\Omega \in \mathcal{K}_2,\,\text{diam}(\Omega)=d \Big\} =: m. \]
First, we observe that
\[ m \geq \frac{3}{4} \cdot \frac{\pi^2}{d^2} > 0\]
by \eqref{lowerboundjung}. Let $\{\Omega_n\}_{n \in \N}$ be a minimizing sequence. Since the functional $\Omega \mapsto \mu_1(\Omega)$ is translation invariant, we can suppose, without loss of generality, that all the sequence is contained in a fixed ball $B_r \subset \R^N$. By Blaschke's selection principle, there exists $\Omega \in \mathcal{K}_N$ and a subsequence (still denoted by $\{\Omega_n\}_{n \in \N}$) such that
\[ \Omega_n \stackrel{\mathcal{H}}{\to} \Omega\]
in the Hausdorff distance. If $|\Omega_n| \to 0$, by Lemma \ref{shrinking} it would hold 
\[ m = \frac{\pi^2}{d^2},\] which would contradict Proposition \ref{reuleaux} if we knew that the eigenfunction in the Reuleaux triangle is smooth. Therefore, we would have $|\Omega|>0$. Now, Proposition \ref{continuityhausdorff} implies that,
\[ \lim_{n \to +\infty} \mu_1(\Omega_n) = \mu_1(\Omega)=m.\]
Moreover, by continuity of the diameter with respect to the Hausdorff convergence of convex sets (see, for instance, \cite{antunesbogosel}), it holds
\[ \text{diam}(\Omega) =d\]
so that $\Omega$ would be a minimizer. In this case, by arguing as in \cite[Theorem 2.1]{bogoselhenrotlucardesi}, it is possible to show the existence of a minimizer which is a \emph{body of convex width}. We conjecture that the Reuleaux triangle is a minimizer, since it is generated from a regular polygon, and among all Reuleaux polygons it is the "farthest" from being a ball.

It is possible to perform similar reasonings for the minimization problem
\[\inf \Big\{\mu_1(\Omega)\,|\,\Omega \in \mathcal{K}_N,\,P(\Omega)=c \Big\}.\]
In this case, we conjecture that the infimum is not attained, and that a minimizing sequence is given by any sequence of rectangles as in Proposition \ref{shrinking}.

\end{document}